\documentclass[12pt,reqno,oneside]{amsart}
\usepackage{amsfonts}
\usepackage{amsmath, amssymb}
\usepackage{hyperref}

\numberwithin{equation}{section} 
\newtheorem{theorem}{Theorem}[section]
\newtheorem{corollary}[theorem]{Corollary}
\newtheorem{lemma}[theorem]{Lemma}
\theoremstyle{definition}
\newtheorem{definition}[theorem]{Definition}
\theoremstyle{remark}
\newtheorem{remark}[theorem]{Remark}
\numberwithin{equation}{section}

\input{tcilatex}

\begin{document}
\title[bi-univalent functions]{Initial coefficient bounds for a general
class of bi-univalent functions}
\author{H. Orhan$^{\ast }$}
\address{Department of Mathematics, Faculty of Science, Ataturk University,
25240 Erzurum, Turkey}
\email{orhanhalit607@gmail.com; horhan@atauni.edu.tr}
\author{N. Magesh}
\address{Post-Graduate and Research Department of Mathematics, Government
Arts College for Men, Krishnagiri 635001, Tamilnadu, India\\
}
\email{nmagi\_2000@yahoo.co.in}
\author{V.K.Balaji}
\address{Department of Mathematics, L.N. Govt College, Ponneri, Chennai,
Tamilnadu, India}
\email{balajilsp@yahoo.co.in}

\begin{abstract}
Inspired by the recent works of Srivastava et al. \cite{HMS-AKM-PG}, Frasin
and Aouf \cite{BAF-MKA} and others \cite%
{Ali-Ravi-Ma-Mina-class,Caglar-Orhan,Goyal-Goswami,Xu-HMS-AML,Xu-HMS-AMC},
we propose to investigate the coefficient estimates for a general class of
analytic and bi-univalent functions. Also, we obtain estimates on the
coefficients $|a_{2}|$ and $|a_{3}|$ for functions in this new class. Some
interesting remarks, corollaries and applications of the results presented
here are also discussed.\ \newline
\end{abstract}

\keywords{Analytic functions, univalent functions, bi-univalent functions,
starlike and convex functions, bi-starlike and bi-convex functions.\\
$^{\ast }$Corresponding author e-mail: orhanhalit607@gmail.com}
\maketitle


\section{Introduction}

Let $\mathcal{A}$ denote the class of functions of the form 
\begin{equation}
f(z)=z+\sum\limits_{n=2}^{\infty }a_{n}z^{n}  \label{Int-e1}
\end{equation}%
which are analytic in the open unit disk $\mathbb{U}=\{z:z\in \mathbb{C}\,\,%
\mathrm{and}\,\,|z|<1\}.$ Further, by $\mathcal{S}$ we shall denote the
class of all functions in $\mathcal{A}$ which are univalent in $\mathbb{U}.$

For two functions $f$ and $g,$ analytic in $\mathbb{U},$ we say that the
function $f(z)$ is subordinate to $g(z)$ in $\mathbb{U},$ and write 
\begin{equation*}
f(z) \prec g(z), \quad z\in \mathbb{U},
\end{equation*}
if there exists a Schwarz function $w(z),$ analytic in $\mathbb{U},$ with 
\begin{equation*}
w(0)=0 \quad \mathrm{and} \quad |w(z)|<1,\quad z\in \mathbb{U},
\end{equation*}
such that 
\begin{equation*}
f(z)=g(w(z)), \quad z\in\mathbb{U}.
\end{equation*}
In particular, if the function $g$ is univalent in $\mathbb{U},$ the above
subordination is equivalent to 
\begin{equation*}
f(0)=g(0) \quad \mathrm{and} \quad f(\mathbb{U}) \subset g(\mathbb{U}).
\end{equation*}

Some of the important and well-investigated subclasses of the univalent
function class $\mathcal{S}$ include (for example) the class $\mathcal{S}%
^*(\alpha)$ of starlike functions of order $\alpha$ in $\mathbb{U}$ and the
class $\mathcal{K}(\alpha)$ of convex functions of order $\alpha$ in $%
\mathbb{U}.$ By definition, we have 
\begin{equation}  \label{ST-e}
\mathcal{S}^*(\alpha):= \left \{ f : f \in \mathcal{S} \, \mathrm{and}\, \Re
\left ( \frac{zf^{\prime }(z)}{f(z)}\right ) > \alpha ; \,\, z \in \mathbb{U}%
;\,\,\, 0 \leq \alpha < 1 \right \}
\end{equation}
and 
\begin{equation}  \label{CV-e}
\mathcal{K}(\alpha):= \left \{ f : f \in \mathcal{S} \, \mathrm{and}\, \Re
\left ( 1+\frac{zf^{\prime \prime }(z)}{f^{\prime }(z)}\right ) > \alpha ;
\,\, z \in \mathbb{U};\,\,\, 0 \leq \alpha < 1 \right \}.
\end{equation}

It readily follows from the definitions (\ref{ST-e}) and (\ref{CV-e}) that 
\begin{equation*}
f \in \mathcal{K}(\alpha) \Longleftrightarrow zf^{\prime }\in \mathcal{S}%
^*(\alpha).
\end{equation*}
Also, let $\mathcal{S}_{\mathcal{P}}^{\beta}(\alpha)$ and $\mathcal{C}_{%
\mathcal{P}}^{\beta}(\alpha)$ denote the subclasses of $\mathcal{S}$
consisting functions $f(z)$ which are defined, respectively by 
\begin{equation}  \label{SP-ST-e}
\mathcal{S}_{\mathcal{P}}^{\beta}(\alpha):= \left \{ f : f \in \mathcal{S}
\, \mathrm{and}\, \Re \left ( e^{i\beta}\frac{zf^{\prime }(z)}{f(z)}\right )
> \alpha\cos\beta ; \,\, z \in \mathbb{U};\,\,\, 0 \leq \alpha < 1,\,\,
\beta \in (-\frac{\pi}{2},\frac{\pi}{2}) \right \}
\end{equation}
and 
\begin{equation}  \label{SP-CV-e}
\mathcal{C}_{\mathcal{P}}^{\beta}(\alpha):= \left \{ f : f \in \mathcal{S}
\, \mathrm{and}\, \Re \left ( e^{i\beta}\frac{z(f^{\prime }(z))^{\prime }}{%
f^{\prime }(z)}\right ) > \alpha\cos\beta ; \,\, z \in \mathbb{U};\,\,\, 0
\leq \alpha < 1,\,\, \beta \in (-\frac{\pi}{2},\frac{\pi}{2}) \right \}.
\end{equation}
It is easy to see that 
\begin{equation*}
f \in \mathcal{C}_{\mathcal{P}}^{\beta}(\alpha) \Longleftrightarrow
zf^{\prime }\in \mathcal{S}_{\mathcal{P}}^{\beta}(\alpha).
\end{equation*}

It is well known that every function $f\in \mathcal{S}$ has an inverse $%
f^{-1},$ defined by 
\begin{equation*}
f^{-1}(f(z))=z,\,\, z \in \mathbb{U}
\end{equation*}
and 
\begin{equation*}
f(f^{-1}(w))=w, \,\, |w| < r_0(f);\,\, r_0(f) \geq \frac{1}{4},
\end{equation*}
where 
\begin{equation}  \label{Int-f-inver}
f^{-1}(w) = w - a_2w^2 + (2a_2^2-a_3)w^3 - (5a_2^3-5a_2a_3+a_4)w^4+\ldots .
\end{equation}

A function $f \in \mathcal{A}$ is said to be bi-univalent in $\mathbb{U}$ if
both $f(z)$ and $f^{-1}(z)$ are univalent in $\mathbb{U}.$ Let $\Sigma$
denote the class of bi-univalent functions in $\mathbb{U}$ given by (\ref%
{Int-e1}). Examples of functions in the class $\Sigma$ are 
\begin{equation*}
\frac{z}{1-z},\,\,\, -\log (1-z),\,\,\, \frac{1}{2}\log \left (\frac{1+z}{1-z%
} \right )
\end{equation*}
and so on. However, the familiar Koebe function is not a member of $\Sigma.$
Other common examples of functions in $\mathcal{S}$ such as 
\begin{equation*}
z-\frac{z^2}{2}\,\,\, \mathrm{and} \,\,\, \frac{z}{1-z^2}
\end{equation*}
are also not members of $\Sigma$ (see \cite{BAF-MKA,HMS-AKM-PG}).

In 1967, Lewin \cite{Lewin} investigated the bi-univalent function class $%
\Sigma$ and showed that $|a_2|<1.51.$ On the other hand, Brannan and Clunie 
\cite{Bran-1979} (see also \cite{Branna1970,Bran1985,Taha1981}) and
Netanyahu \cite{Netany} made an attempt to introduce various subclasses of
the bi-univalent function class $\Sigma$ and obtained non-sharp coefficient
estimates on the first two coefficients $|a_2|$ and $|a_3|$ of (\ref{Int-e1}%
). But the coefficient problem for each of the following Taylor-Maclaurin
coefficients $|a_n|\, (n\in\mathbb{N}\setminus\{1,2\};\;\;\mathbb{N}%
:=\{1,2,3,\cdots\})$ is still an open problem. Following Brannan and Taha 
\cite{Bran1985}, many researchers (see\cite%
{Ali-Ravi-Ma-Mina-class,Caglar-Orhan,BAF-MKA,Goyal-Goswami,haya,Li-Wang,SSS-VR-VR,HMS-AKM-PG,Xu-HMS-AML,Xu-HMS-AMC}%
) have recently introduced and investigated several interesting subclasses
of the bi-univalent function class $\Sigma$ and they have found non-sharp
estimates on the first two Taylor-Maclaurin coefficients $|a_2|$ and $|a_3|.$

Motivated by the above mentioned works, we define the following subclass of
function class $\Sigma.$

\begin{definition}
Let $h:\mathbb{U}\rightarrow \mathbb{C},$ be a convex univalent function
such that $h(0)=1$ and $h(\bar{z})=\overline{h(z)},$ for $z\in \mathbb{U}$
and $\Re(h(z))>0.$ A function $f(z)$ given by (\ref{Int-e1}) is said to be
in the class $\mathcal{N}\mathcal{P}^{\mu, \lambda}_{\Sigma}(\beta,h)$ if
the following conditions are satisfied: 
\begin{equation}  \label{Defi-2-e1}
f \in \Sigma, \,\, e^{i\beta}\left ( (1-\lambda)\left(\frac{f(z)}{z}%
\right)^{\mu}+\lambda f^{\prime }(z)\left(\frac{f(z)}{z}\right)^{\mu-1}
\right )\prec h(z) \cos\beta+i\sin\beta,\,\, \,\, z \in \mathbb{U}
\end{equation}
and 
\begin{equation}  \label{Defi-2-e2}
e^{i\beta}\left ( (1-\lambda)\left(\frac{g(w)}{w}\right)^{\mu}+\lambda
g^{\prime }(w)\left(\frac{g(w)}{w}\right)^{\mu-1} \right )\prec h(w)
\cos\beta+i\sin\beta,\,\, \,\, w\in \mathbb{U},
\end{equation}
where $\beta\in(-\pi/2,\pi/2),$ $\lambda \geq 1,$ $\mu\geq 0$ and the
function $g$ is given by 
\begin{equation}  \label{g-e}
g(w) = w - a_2w^2 + (2a_2^2-a_3)w^3 - (5a_2^3-5a_2a_3+a_4)w^4+\ldots
\end{equation}
the extension of $f^{-1}$ to $\mathbb{U}.$
\end{definition}

\begin{remark}
\label{Bi-R-1} If we set $h(z)=\frac{1+Az}{1+Bz},$ $-1\leq B <A\leq 1,$ in
the class $\mathcal{N}\mathcal{P}^{\mu, \lambda}_{\Sigma}(\beta,h),$ we have 
$\mathcal{N}\mathcal{P}^{\mu, \lambda}_{\Sigma}(\beta,\frac{1+Az}{1+Bz})$
and defined as 
\begin{equation*}
f \in \Sigma, \,\, e^{i\beta}\left ( (1-\lambda)\left(\frac{f(z)}{z}%
\right)^{\mu}+\lambda f^{\prime }(z)\left(\frac{f(z)}{z}\right)^{\mu-1}
\right )\prec \frac{1+Az}{1+Bz} \cos\beta+i\sin\beta,\,\, \,\, z \in \mathbb{%
U}
\end{equation*}
and 
\begin{equation*}
e^{i\beta}\left ( (1-\lambda)\left(\frac{g(w)}{w}\right)^{\mu}+\lambda
g^{\prime }(w)\left(\frac{g(w)}{w}\right)^{\mu-1} \right )\prec \frac{1+Aw}{%
1+Bw} \cos\beta+i\sin\beta,\,\, \,\, w\in \mathbb{U},
\end{equation*}
where $\beta\in(-\pi/2,\pi/2),$ $\lambda \geq 1,$ $\mu\geq 0$ and the
function $g$ is given by (\ref{g-e}).
\end{remark}

\begin{remark}
\label{Bi-R-5} Taking $h(z)=\frac{1+(1-2\alpha)z}{1-z},$ $0\leq\alpha<1$ in
the class $\mathcal{N}\mathcal{P}^{\mu, \lambda}_{\Sigma}(\beta,h),$ we have 
$\mathcal{N}\mathcal{P}^{\mu,\lambda}_{\Sigma}(\beta,\alpha)$ and $f\in%
\mathcal{N}\mathcal{P}^{\mu,\lambda}_{\Sigma}(\beta, \alpha)$ if the
following conditions are satisfied: 
\begin{equation*}
f \in \Sigma, \,\, \Re \left (e^{i\beta}\left ( (1-\lambda)\left(\frac{f(z)}{%
z}\right)^{\mu}+\lambda f^{\prime }(z)\left(\frac{f(z)}{z}\right)^{\mu-1}
\right )\right ) > \alpha\cos\beta,\,\, 0 \leq \alpha < 1;\,\, z \in \mathbb{%
U}
\end{equation*}
and 
\begin{equation*}
\Re \left (e^{i\beta}\left ( (1-\lambda)\left(\frac{g(w)}{w}%
\right)^{\mu}+\lambda g^{\prime }(w)\left(\frac{g(w)}{w}\right)^{\mu-1}
\right )\right )>\alpha\cos\beta,\,\, 0 \leq \alpha < 1;\,\, w \in \mathbb{U}%
,
\end{equation*}
where $\beta\in(-\pi/2,\pi/2),$ $\lambda \geq 1,$ $\mu\geq 0$ and the
function $g$ is given by (\ref{g-e}).
\end{remark}

\begin{remark}
\label{Bi-R-6} Taking $\lambda=1$ and $h(z)=\frac{1+(1-2\alpha)z}{1-z},$ $%
0\leq\alpha<1$ in the class $\mathcal{N}\mathcal{P}^{\mu,
\lambda}_{\Sigma}(\beta,h),$ we have $\mathcal{N}\mathcal{P}%
^{\mu,1}_{\Sigma}(\beta,\alpha)$ and $f\in\mathcal{N}\mathcal{P}%
^{\mu,1}_{\Sigma}(\beta, \alpha)$ if the following conditions are satisfied: 
\begin{equation*}
f \in \Sigma, \,\, \Re \left (e^{i\beta}f^{\prime }(z)\left (\frac{f(z)}{z}%
\right )^{\mu-1} \right )> \alpha\cos\beta,\,\, 0 \leq \alpha <
1;\,\,\mu\geq 0;\,\, z \in \mathbb{U}
\end{equation*}
and 
\begin{equation*}
\Re \left (e^{i\beta}g^{\prime }(w)\left (\frac{g(w)}{w}\right )^{\mu-1}
\right )>\alpha\cos\beta,\,\, 0 \leq \alpha < 1;\,\,\mu\geq 0;\,\, w \in 
\mathbb{U},
\end{equation*}
where $\beta\in(-\pi/2,\pi/2)$ and the function $g$ is given by (\ref{g-e}).
\end{remark}

\begin{remark}
\label{Bi-R-7} Taking $\mu+1=\lambda=1$ and $h(z)=\frac{1+(1-2\alpha)z}{1-z}%
, $ $0\leq\alpha<1$ in the class $\mathcal{N}\mathcal{P}^{\mu,
\lambda}_{\Sigma}(\beta,h),$ we have $\mathcal{N}\mathcal{P}%
^{0,1}_{\Sigma}(\beta,\alpha)$ and $f\in\mathcal{N}\mathcal{P}%
^{0,1}_{\Sigma}(\beta, \alpha)$ if the following conditions are satisfied: 
\begin{equation*}
f \in \Sigma, \,\, \Re \left (e^{i\beta}\frac{zf^{\prime }(z)}{f(z)}\right )
> \alpha\cos\beta,\,\, 0 \leq \alpha < 1;\,\, z \in \mathbb{U}
\end{equation*}
and 
\begin{equation*}
\Re \left (e^{i\beta}\frac{wg^{\prime }(w)}{g(w)}\right
)>\alpha\cos\beta,\,\, 0 \leq \alpha < 1;\,\, w \in \mathbb{U},
\end{equation*}
where $\beta\in(-\pi/2,\pi/2)$ and the function $g$ is given by (\ref{g-e}).
\end{remark}

\begin{remark}
\label{Bi-R-8} Taking $\mu=1$ and $h(z)=\frac{1+(1-2\alpha)z}{1-z},$ $%
0\leq\alpha<1$ in the class $\mathcal{N}\mathcal{P}^{\mu,
\lambda}_{\Sigma}(\beta,h),$ we have $\mathcal{N}\mathcal{P}%
^{1,\lambda}_{\Sigma}(\beta,\alpha)$ and $f\in\mathcal{N}\mathcal{P}%
^{1,\lambda}_{\Sigma}(\beta, \alpha)$ if the following conditions are
satisfied: 
\begin{equation*}
f \in \Sigma, \,\, \Re \left (e^{i\beta}\left ( (1-\lambda)\frac{f(z)}{z}%
+\lambda f^{\prime }(z)\right )\right ) > \alpha\cos\beta,\,\, 0 \leq \alpha
< 1;\,\,\lambda \geq 1;\,\, z \in \mathbb{U}
\end{equation*}
and 
\begin{equation*}
\Re \left (e^{i\beta}\left ( (1-\lambda)\frac{g(w)}{w}+\lambda g^{\prime
}(w)\right )\right)>\alpha\cos\beta,\,\, 0 \leq \alpha < 1;\,\,\lambda \geq
1;\,\, w \in \mathbb{U},
\end{equation*}
where $\beta\in(-\pi/2,\pi/2)$ and the function $g$ is given by (\ref{g-e}).
\end{remark}

\begin{remark}
\label{Bi-R-8a} Taking $\mu=\lambda=1$ and $h(z)=\frac{1+(1-2\alpha)z}{1-z},$
$0\leq\alpha<1$ in the class $\mathcal{N}\mathcal{P}^{\mu,
\lambda}_{\Sigma}(\beta,h),$ we have $\mathcal{N}\mathcal{P}%
^{1,1}_{\Sigma}(\beta,\alpha)$ and $f\in\mathcal{N}\mathcal{P}%
^{1,1}_{\Sigma}(\beta, \alpha)$ if the following conditions are satisfied: 
\begin{equation*}
f \in \Sigma, \,\, \Re \left (e^{i\beta}f^{\prime }(z)\right ) >
\alpha\cos\beta,\,\, 0 \leq \alpha < 1;\,\, z \in \mathbb{U}
\end{equation*}
and 
\begin{equation*}
\Re \left (e^{i\beta}g^{\prime }(w)\right)>\alpha\cos\beta,\,\, 0 \leq
\alpha < 1;\,\, w \in \mathbb{U},
\end{equation*}
where $\beta\in(-\pi/2,\pi/2)$ and the function $g$ is given by (\ref{g-e}).
\end{remark}

We note that

\begin{enumerate}
\item $\mathcal{N}\mathcal{P}^{1,1}_{\Sigma}(0,\alpha)$ = $\mathcal{H}%
_{\Sigma}^{\alpha}$ \,\,(see \cite{HMS-AKM-PG})

\item $\mathcal{N}\mathcal{P}^{1,\lambda}_{\Sigma}(0,\alpha)$ = $\mathcal{B}%
_{\Sigma}(\alpha,\lambda)$\,\,(see \cite{BAF-MKA})

\item $\mathcal{N}\mathcal{P}^{0,1}_{\Sigma}(0,\alpha)$ = $\mathcal{F}%
_{\Sigma}(\alpha)$ \,\, (see \cite{Li-Wang})

\item $\mathcal{N}\mathcal{P}^{\mu,1}_{\Sigma}(0,\alpha)$ = $\mathcal{N}%
^{\mu}_{\Sigma}(\alpha)$ \,\, (see \cite{SSS-VR-VR})

\item $\mathcal{N}\mathcal{P}^{\mu,\lambda}_{\Sigma}(0,\alpha)$ = $\mathcal{N%
}^{\mu,\lambda}_{\Sigma}(\alpha)$ \,\, (see \cite{Caglar-Orhan}).
\end{enumerate}

In order to derive our main result, we have to recall here the following
lemma.

\begin{lemma}
\label{lem-pom-2}\cite{Rogosiniki,Xu-HMS-AML-11} Let the function $%
\varphi(z) $ given by 
\begin{equation*}
\varphi(z)=\sum\limits_{n=1}^{\infty}B_nz^n,\quad z\in\mathbb{U}
\end{equation*}
be convex in $\mathbb{U}.$ Suppose also that the function $h(z)$ given by 
\begin{equation*}
h(z)=\sum\limits_{n=1}^{\infty}h_nz^n,\quad z\in\mathbb{U}
\end{equation*}
is holomorphic in $\mathbb{U}.$ If $h(z)\prec \varphi(z),$ $z \in \mathbb{U}%
, $ then $|h_n| \leq |B_1|,$ $n\in \mathbb{N}=\{1,2,3,\dots\}.$
\end{lemma}

The object of the present paper is to introduce a general new subclass $%
\mathcal{N}\mathcal{P}^{\mu, \lambda}_{\Sigma}(\beta,h)$ of the function
class $\Sigma$ and obtain estimates of the coefficients $|a_2|$ and $|a_3|$
for functions in this new class $\mathcal{N}\mathcal{P}^{\mu,
\lambda}_{\Sigma}(\beta,h).$

\section{Coefficient bounds for the function class $\mathcal{N}\mathcal{P}^{%
\protect\mu, \protect\lambda}_{\Sigma}(\protect\beta,h)$}

In this section we find the estimates on the coefficients $|a_2|$ and $|a_3|$
for functions in the class $\mathcal{N}\mathcal{P}^{\mu,
\lambda}_{\Sigma}(\beta,h).$

\begin{theorem}
\label{Bi-th2} Let $f(z)$ given by (\ref{Int-e1}) be in the class $\mathcal{N%
}\mathcal{P}^{\mu,\lambda}_{\Sigma}(\beta, h),$ $0 \leq \alpha < 1,$ $%
\lambda \geq 1$ and $\mu geq 0,$ then 
\begin{equation}  \label{bi-th2-b-a2}
|a_2| \leq \sqrt{\frac{2|B_1|\cos\beta}{(1+\mu)(2\lambda+\mu)}}
\end{equation}
and 
\begin{equation}  \label{bi-th2-b-a3}
|a_3| \leq \frac{2|B_1|cos\beta}{(2\lambda+\mu)(1+\mu)},
\end{equation}
where $\beta\in(-\pi/2,\pi/2).$
\end{theorem}

\begin{proof}
It follows from (\ref{Defi-2-e1}) and (\ref{Defi-2-e2}) that there exists $%
p,q \in \mathcal{P}$ such that 
\begin{equation}  \label{bi-th2-pr-e1}
e^{i\beta}\left ( (1-\lambda)\left(\frac{f(z)}{z}\right)^{\mu}+\lambda
f^{\prime }(z)\left(\frac{f(z)}{z}\right)^{\mu-1} \right )= p(z)\cos\beta +
i \sin\beta
\end{equation}
and 
\begin{equation}  \label{bi-th2-pr-e2}
e^{i\beta}\left ( (1-\lambda)\left(\frac{g(w)}{w}\right)^{\mu}+\lambda
g^{\prime }(w)\left(\frac{g(w)}{w}\right)^{\mu-1} \right )= p(w)\cos\beta +
i \sin\beta,
\end{equation}
where $p(z)\prec h(z)$ and $q(w)\prec h(w)$ have the forms 
\begin{equation}  \label{Exp-p(z)}
p(z)=1+p_1z+p_2z^2+\dots,\,\, z \in \mathbb{U}
\end{equation}
and 
\begin{equation}  \label{Exp-q(w)}
q(z)=1+q_1w+q_2w^2+\dots,\,\, w \in \mathbb{U} .
\end{equation}
Equating coefficients in (\ref{bi-th2-pr-e1}) and (\ref{bi-th2-pr-e2}), we
get 
\begin{equation}  \label{th2-ceof-p1}
e^{i\beta}(\lambda+\mu)a_2 = p_1\cos\beta
\end{equation}
\begin{equation}  \label{th2-ceof-p2}
e^{i\beta}\left[\frac{a_2^2}{2}(\mu-1)+a_3\right](2\lambda+\mu) =
p_2\cos\beta
\end{equation}
\begin{equation}  \label{th2-ceof-q1}
-e^{i\beta}(\lambda+\mu)a_2 = q_1\cos\beta
\end{equation}
and 
\begin{equation}  \label{th2-ceof-q2}
e^{i\beta}\left[(\mu+3)\frac{a_2^2}{2}-a_3\right](2\lambda+\mu) =
q_2\cos\beta .
\end{equation}
From (\ref{th2-ceof-p1}) and (\ref{th2-ceof-q1}), we get 
\begin{equation}  \label{th2-pr-p1=q1}
p_1=-q_1
\end{equation}
and 
\begin{equation}  \label{th2-pr-p1=q1N}
2e^{i\beta}(\lambda+\mu)^2a_2^2 = (p_1^2+q_1^2)\cos^2\beta .
\end{equation}
Also, from (\ref{th2-ceof-p2}) and (\ref{th2-ceof-q2}), we obtain 
\begin{equation}  \label{a2-square}
a_2^2=\frac{e^{-i\beta}(p_2+q_2)\cos\beta}{(1+\mu)(2\lambda+\mu)}.
\end{equation}
Since $p,q\in h(\mathbb{U}),$ applying Lemma \ref{lem-pom-2}, we immediately
have 
\begin{equation}  \label{p-k-B-1}
|p_m|=\left|\frac{p^{(m)}(0)}{m!}\right|\leq|B_1|,\,\,m\in\mathbb{N},
\end{equation}
and 
\begin{equation}  \label{q-k-B-1}
|q_m|=\left|\frac{q^{(m)}(0)}{m!}\right|\leq|B_1|,\,\,m\in\mathbb{N}.
\end{equation}
Applying (\ref{p-k-B-1}), (\ref{q-k-B-1}) and Lemma \ref{lem-pom-2} for the
coefficients $p_1,$ $p_2,$ $q_1$ and $q_2,$ we readily get

\begin{equation*}
|a_2|\leq\sqrt{\frac{2|B_1|\cos\beta}{(1+\mu)(2\lambda+\mu)}}.
\end{equation*}
This gives the bound on $|a_2|$ as asserted in (\ref{bi-th2-b-a2}).

Next, in order to find the bound on $|a_3|$, by subtracting (\ref%
{th2-ceof-q2}) from (\ref{th2-ceof-p2}), we get 
\begin{equation}  \label{th2-a3-cal-e1}
2(a_3 - a_2^2)(2\lambda+\mu) = e^{-i\beta}(p_2-q_2)\cos\beta .
\end{equation}
It follows from (\ref{a2-square}) and (\ref{th2-a3-cal-e1}) that 
\begin{equation}  \label{th2-a3-cal-e2}
a_3 = \frac{e^{-i\beta}\cos\beta(p_2+q_2)}{(1+\mu)(2\lambda+\mu)}+\frac{%
e^{-i\beta}(p_2-q_2)\cos\beta}{2(2\lambda+\mu)}.
\end{equation}
Applying (\ref{p-k-B-1}), (\ref{q-k-B-1}) and Lemma \ref{lem-pom-2} once
again for the coefficients $p_1,$ $p_2,$ $q_1$ and $q_2,$ we readily get 
\begin{equation*}
|a_3| \leq \frac{2|B_1|cos\beta}{(2\lambda+\mu)(1+\mu)}.
\end{equation*}
This completes the proof of Theorem \ref{Bi-th2}.
\end{proof}


\section{Corollaries and Consequences}

In view of Remark \ref{Bi-R-1}, if we set 
\begin{equation*}
h(z)=\frac{1+Az}{1+Bz},\quad -1\leq B < A \leq 1,\, z\in\mathbb{U}
\end{equation*}
and 
\begin{equation*}
h(z)=\frac{1+(1-2\alpha)z}{1-z},\,\, 0\leq \alpha <1,\,\, z\in \mathbb{U},
\end{equation*}
in Theorem \ref{Bi-th2}, we can readily deduce Corollaries \ref{Bi-th2-Cor1}
and \ref{Bi-th2-Cor2}, respectively, which we merely state here without
proof.

\begin{corollary}
\label{Bi-th2-Cor1} Let $f(z)$ given by (\ref{Int-e1}) be in the class $%
\mathcal{N}\mathcal{P}^{\mu,\lambda}_{\Sigma}(\beta,\frac{1+Az}{1+Bz}),$
then 
\begin{equation}  \label{bi-th2-Cor1-b-a2}
|a_2| \leq \sqrt{\frac{2(A-B)\cos\beta}{(1+\mu)(2\lambda+\mu)}}
\end{equation}
and 
\begin{equation}  \label{bi-th2-Cor1-b-a3}
|a_3| \leq \frac{2(A-B)cos\beta}{(2\lambda+\mu)(1+\mu)},
\end{equation}
where $\beta\in(-\pi/2,\pi/2),$ $\mu \geq 0$ and $\lambda \geq 1.$
\end{corollary}

\begin{corollary}
\label{Bi-th2-Cor2} Let $f(z)$ given by (\ref{Int-e1}) be in the class $%
\mathcal{N}\mathcal{P}^{\mu,\lambda}_{\Sigma}(\beta,\alpha),$ $0\leq \alpha
< 1,$ $\mu \geq 0$ and $\lambda \geq 1,$ then 
\begin{equation}  \label{bi-th2-Cor2-b-a2}
|a_2| \leq \sqrt{\frac{4(1-\alpha)\cos\beta}{(1+\mu)(2\lambda+\mu)}}
\end{equation}
and 
\begin{equation}  \label{bi-th2-Cor2-b-a3}
|a_3| \leq \frac{4(1-\alpha)cos\beta}{(2\lambda+\mu)(1+\mu)},
\end{equation}
where $\beta\in(-\pi/2,\pi/2).$
\end{corollary}

\begin{remark}
When $\beta=0$ the estimates of the coefficients $|a_2|$ and $|a_3|$ of the
Corollary \ref{Bi-th2-Cor2} are improvement of the estimates obtained in 
\cite[Theorem 3.1]{Caglar-Orhan}.
\end{remark}

\begin{corollary}
\label{Bi-th2-Cor3} Let $f(z)$ given by (\ref{Int-e1}) be in the class $%
\mathcal{N}\mathcal{P}^{\mu,1}_{\Sigma}(\beta,\alpha),$ $0\leq \alpha < 1$
and $\mu \geq 0,$ then 
\begin{equation}  \label{bi-th2-Cor3-b-a2}
|a_2| \leq \sqrt{\frac{4(1-\alpha)\cos\beta}{(1+\mu)(2+\mu)}}
\end{equation}
and 
\begin{equation}  \label{bi-th2-Cor3-b-a3}
|a_3| \leq \frac{4(1-\alpha)cos\beta}{(2+\mu)(1+\mu)},
\end{equation}
where $\beta\in(-\pi/2,\pi/2).$
\end{corollary}

\begin{corollary}
\label{Bi-th2-Cor4} Let $f(z)$ given by (\ref{Int-e1}) be in the class $%
\mathcal{N}\mathcal{P}^{0,1}_{\Sigma}(\beta,\alpha),$ $0\leq \alpha < 1,$
then 
\begin{equation}  \label{bi-th2-Cor4-b-a2}
|a_2| \leq \sqrt{2(1-\alpha)\cos\beta}
\end{equation}
and 
\begin{equation}  \label{bi-th2-Cor4-b-a3}
|a_3| \leq 2(1-\alpha)cos\beta,
\end{equation}
where $\beta\in(-\pi/2,\pi/2).$
\end{corollary}

\begin{remark}
Taking $\beta=0$ in Corollary \ref{Bi-th2-Cor4}, the estimate (\ref%
{bi-th2-Cor4-b-a2}) reduces to $|a_2|$ of \cite[Corollary 3.3]{Li-Wang} and (%
\ref{bi-th2-Cor4-b-a3}) is improvement of $|a_3|$ obtained in \cite[%
Corollary 3.3]{Li-Wang}.
\end{remark}

\begin{corollary}
\label{Bi-th2-Cor5} Let $f(z)$ given by (\ref{Int-e1}) be in the class $%
\mathcal{N}\mathcal{P}^{1,\lambda}_{\Sigma}(\beta,\alpha),$ $0\leq \alpha <
1 $ and $\lambda \geq 1,$ then 
\begin{equation}  \label{bi-th2-Cor5-b-a2}
|a_2| \leq \sqrt{\frac{2(1-\alpha)\cos\beta}{2\lambda+1}}
\end{equation}
and 
\begin{equation}  \label{bi-th2-Cor5-b-a3}
|a_3| \leq \frac{2(1-\alpha)cos\beta}{2\lambda+1},
\end{equation}
where $\beta\in(-\pi/2,\pi/2).$
\end{corollary}

\begin{remark}
Taking $\beta=0$ in Corollary \ref{Bi-th2-Cor5}, the inequality (\ref%
{bi-th2-Cor5-b-a3}) improves the estimate of $|a_3|$ in \cite[Theorem 3.2]%
{BAF-MKA}.
\end{remark}

\begin{corollary}
\label{Bi-th2-Cor6} Let $f(z)$ given by (\ref{Int-e1}) be in the class $%
\mathcal{N}\mathcal{P}^{1,1}_{\Sigma}(\beta,\alpha),$ $0\leq \alpha <1,$
then 
\begin{equation}  \label{bi-th2-Cor6-b-a2}
|a_2| \leq \sqrt{\frac{2(1-\alpha)\cos\beta}{3}}
\end{equation}
and 
\begin{equation}  \label{bi-th2-Cor6-b-a3}
|a_3| \leq \frac{2(1-\alpha)cos\beta}{3},
\end{equation}
where $\beta\in(-\pi/2,\pi/2).$
\end{corollary}

\begin{remark}
For $\beta=0$ the inequality (\ref{bi-th2-Cor6-b-a3}) improves the estimate $%
|a_3|$ of \cite[Theorem 2]{HMS-AKM-PG}.
\end{remark}


\end{document}